\documentclass[12pt]{article}
\usepackage{mathtools,amsthm,amsmath,amsfonts,amssymb,tikz-cd,enumitem, graphicx,mathrsfs,bbm,stmaryrd,xcolor}

\tikzcdset{scale cd/.style={every label/.append style={scale=#1},
    cells={nodes={scale=#1}}}}
\usepackage{authblk}
\usepackage{url}
\usepackage{esvect}
\makeatletter
\newcommand{\address}[1]{\gdef\@address{#1}}
\newcommand{\email}[1]{\gdef\@email{\url{#1}}}
\newcommand{\@endstuff}{\par\vspace{\baselineskip}\noindent\small
\begin{tabular}{@{}l}\@address\\\textit{E-mail address:} \@email\end{tabular}}
\AtEndDocument{\@endstuff}
\makeatother
\usepackage{setspace}
\usepackage[utf8]{inputenc}
\usepackage{CJKutf8}
\counterwithin{equation}{section}
\usepackage[pdfencoding=auto,psdextra,colorlinks=true,linkcolor=blue,citecolor=blue,urlcolor = orange]{hyperref}
\usepackage[margin = 1in]{geometry}
\usepackage{kantlipsum}
\allowdisplaybreaks
\newtheorem{theorem}{Theorem}[section]
\newtheorem{definition}[theorem]{Definition}
\newtheorem{example}[theorem]{Example}
\newtheorem{lemma}[theorem]{Lemma}
\newtheorem{question}[theorem]{Question}
\newtheorem{remark}[theorem]{Remark}
\newtheorem{proposition}[theorem]{Proposition}
\newtheorem{corollary}[theorem]{Corollary}
\newtheorem{assumption}[theorem]{Assumption}
\newcommand{\citep}[1]{\cite{#1}}

\newcommand{\professor}{\text{Prof.\ }\hspace{-0.03125mm}}
\newcommand{\doctor}{\text{Dr.\ }\hspace{-0.03125mm}}
\newcommand{\edmp}{\text{\normalfont End}}

\newcommand{\tracebold}{\text{\normalfont\textbf{Tr}}}
\newcommand{\supertrace}{\text{\normalfont Str}}
\newcommand{\supertracebold}{\text{\normalfont\textbf{Str}}}
\begin{document}
\title{\textbf{Transgression in the primitive cohomology}}
\author{Hao Zhuang}
\address{Beijing International Center for Mathematical Research, Peking University}
\email{hzhuang@pku.edu.cn}
\date{December 31, 2025}
\maketitle
\begin{abstract}
We study the Chern-Weil theory for the primitive cohomology of a symplectic manifold. First, given a symplectic manifold, we review the superbundle-valued forms on this manifold and prove a primitive version of the Bianchi identity. Second, as the main result, we prove a transgression formula associated with the boundary map of the primitive cohomology. Third, as an application of the main result, we introduce the concept of primitive characteristic classes and point out a further direction. 
\end{abstract} 
\tableofcontents
\section{Introduction}
The primitive cohomology of a symplectic manifold draws much attention in the past several years. It was first introduced by Tseng and Yau in \cite[(3.14), (3.22)]{tty1st} and \cite[19, (1.5), (1.6)]{tty2nd}. Later, Tsai, Tseng, and Yau constructed the more general $p$-filtered cohomology \cite[(1.2), Theorem 3.1]{tty3rd} and included the primitive cohomology as the $0$-filtered part. The constructions in \cite{tty3rd, tty1st, tty2nd} use the Lefschetz decomposition. An equivalent construction using the mapping cone was given by Tanaka and Tseng \cite[Theorem 1.1]{tanaka_tseng_2018}. 

Different from the de Rham cohomology, the most important point of the primitive cohomology is that it relies on the symplectic form. Thus, by \cite[Section 4]{tty2nd} and \cite[Section 6.3]{tty3rd}, the primitive cohomology can be used to distinguish between different symplectic structures. 

Meanwhile, many topics in geometry and topology naturally extends to the primitive cohomology. In these primitive versions, the symplectic form reveals extra features in both the procedures and the results. For instance, in terms of differential topology, Clausen, Tang, and Tseng developed the symplectic Morse theory \cite[Theorems 1.3-1.4]{tangtsengclausensymplecticwitten} and \cite[Theorems 1.2-1.3]{tangtsengclausenmappingcone}, connecting 
the primitive cohomology to the critical points of Morse functions. In our previous work, we introduced the primitive and the $1$-filtered versions of semi-characteristics \cite[Theorem 1.5]{symplectic_semi_char_2025} and \cite[Theorem 1.2]{1_filtered_semi_char_2025}, relating to zero points of vector fields.

Besides the perspective of differential topology, there are also studies on the primitive cohomology from the perspective of connections and characteristic classes. In \cite[Definition 1.1]{tseng_and_zhou_symplectic_flat_connection_and_twisted_primitive2022}, for connections on a vector bundle over the symplectic manifold, Tseng and Zhou introduced the concept of symplectic flatness. In \cite[Section 4]{tseng_zhou_symplectic_flat_functional_characteristic_classes2022}, they equipped a $G$-bundle with a symplectically flat connection and studied the characteristic classes of the $G$-bundle. More importantly, in \cite[Definitions 1.3-1.4]{tseng_zhou_symplectic_flat_functional_characteristic_classes2022} and \cite[(1.1)]{tseng_and_zhou_2025mapping_yang_mills}, they introduced the primitive Yang-Mills functionals and proved the associated classification theorems \cite[Theorem 1.2]{tseng_zhou_symplectic_flat_functional_characteristic_classes2022} and \cite[Theorem 1.1]{tseng_and_zhou_2025mapping_yang_mills} of symplectically flat $G$-bundles over the symplectic manifold. In addition, as they pointed out, the results in \cite{tseng_zhou_symplectic_flat_functional_characteristic_classes2022, tseng_and_zhou_2025mapping_yang_mills} are true even when the symplectic form is replaced by any nondegenerate closed $2$-form.

Tseng and Zhou's settings of the symplectic flatness involves a connection on the bundle and a smooth section of the associated endomorphism bundle. Motivated by their construction of the primitive version of the Yang-Mills functional, we believe that even when the connection is not symplectically flat, it is still worthwhile to put the connection and the smooth section together as a ``primitive connection''. Afterwards, this ``primitive connection'' induces other geometric or topological objects. 

In this paper, we develop the Chern-Weil theory using this ``primitive connection'' for the primitive cohomology. Once we find the primitive version of the Chern-Weil transgression formula associated with the boundary map of the primitive cohomology, we can immediately define the primitive versions of characteristic classes involving the symplectic form. To make things more unified, we present our results in the language of superbundles \cite[Section 2]{quillen_superconnection}.

We adopt the construction \cite[Theorem 1.1]{tanaka_tseng_2018} of the primitive cohomology. For superbundles and the transgression, we follow the approaches similar to \cite[Sections 1.3-1.5]{bgv}. 

\begin{assumption}\normalfont
   We let $(M,\omega)$ be a symplectic manifold, and $E = E^+\oplus E^-$ be a $\mathbb{Z}_2$-graded smooth superbundle \cite[Definition 1.29]{bgv} over $M$.
\end{assumption}

Let $\Omega^i(M,E)$ (resp. $\Omega(M,E)$) be the space of smooth $E$-valued $i$-forms (resp. all smooth $E$-valued forms) on $M$. Following \cite[Section 3.1]{tanaka_tseng_2018}, we give the map
    \begin{align*}
        \partial: \Omega(M)\oplus\Omega(M)&\to\Omega(M)\oplus\Omega(M)\\
        (\alpha,\beta)&\mapsto (d\alpha+\omega\wedge\beta, -d\beta).
    \end{align*}
It defines a chain complex 
\begin{align}\label{chain complex primitive}
    \partial: \Omega^i(M)\oplus\Omega^{i-1}(M)\to\Omega^{i+1}(M)\oplus\Omega^i(M) \ \ (0\leqslant i\leqslant \dim M+1).
\end{align}
\begin{definition}\label{definition of primitive cohomology}\normalfont
    The cohomology given by (\ref{chain complex primitive})
is called the primitive cohomology of $(M,\omega)$.
\end{definition}

We follow the conventions of signs in \cite[Sections 1.3-1.5]{bgv} and \cite[Sections 1.2-1.3]{weipingzhangnewedition} and then obtain 
$\Omega^\pm(M,E)$ and $\Omega^\pm(M,\edmp(E))$.  

Following  \cite[(1.6)]{tseng_and_zhou_symplectic_flat_connection_and_twisted_primitive2022}, \cite[(1.1)]{tseng_zhou_symplectic_flat_functional_characteristic_classes2022}, and \cite[(1.1)]{tseng_and_zhou_2025mapping_yang_mills}, we define the primitive superconnection. In fact, the primitive superconnection is just a map reversing the parity of the total degree, but it is not a true connection. See Remark \ref{not a connection remark} for explanations. 
\begin{definition}\label{extended superconnection}
    \normalfont We call the linear map
\begin{align*}
    \mathbb{A}: \Omega(M,E)\oplus\Omega(M,E)&\to \Omega(M,E)\oplus\Omega(M,E)\\
       (\alpha, \beta) &\mapsto (A\alpha+\omega\wedge\beta, B\alpha - A\beta)
\end{align*}
a primitive superconnection when $A$ is a superconnection on $E$ and $B\in\Omega^+(M,\edmp(E))$.
\end{definition}
\begin{remark}\normalfont
   The superconnection $A$ on $E$ is symplectically flat if and only if the map $\mathbb{A}^2 = 0$ (cf. \cite[Proposition 3.8]{tseng_and_zhou_symplectic_flat_connection_and_twisted_primitive2022}). See (\ref{A^2 computation}) for the expression of $\mathbb{A}^2$ in our notations.
\end{remark}

We know \cite[Sections 1.3 \& 1.5]{bgv} that for any $\gamma\in\Omega(M,\edmp(E))$, it has a supertrace $\supertrace(\gamma)\in\Omega(M)$. We extend the definition of supertraces as follows.
\begin{definition}\label{supertracebold definition}\normalfont
    For any $$(\gamma,\delta)\in\Omega(M,\edmp(E))\oplus\Omega(M,\edmp(E)),$$ we call 
    $$(\supertrace(\gamma),\supertrace(\delta))\in\Omega(M)\oplus\Omega(M)$$
    the supertrace of $(\gamma,\delta)$ and denote it by $\supertracebold(\gamma,\delta)$. 
\end{definition}

Now, we give the main result, including a $\partial$-closed element and a transgression formula.
\begin{theorem}\label{main result}
    For a smooth family of primitive superconnections $\mathbb{A}_t$ $(t\in\mathbb{R})$ and any $k\in\mathbb{N}$, we identify maps $\mathbb{A}_t^{2k}$ and $\dfrac{d\mathbb{A}_t}{dt}\mathbb{A}_t^{2k-2}$ with unique elements 
$$\mathbb{A}_t^{2k}\in\Omega^{+}(M,\edmp(E))\oplus\Omega^{-}(M,\edmp(E))$$
and $$\dfrac{d\mathbb{A}_t}{dt}\mathbb{A}_t^{2k-2}\in\Omega^{-}(M,\edmp(E))\oplus\Omega^{+}(M,\edmp(E))$$ respectively according to {\normalfont (\ref{A2k expression})} and {\normalfont (\ref{dA_tA_t^{2k} expression})}. Then, for any polynomial $f\in\mathbb{C}[z]$, we have 
    \begin{enumerate}[label = {\normalfont(\arabic*)}]
        \item $\partial\hspace{+0.5mm}\supertracebold(f(\mathbb{A}_t^2)) = (0,0)$,  
        \item $\dfrac{d}{dt}\supertracebold(f(\mathbb{A}_t^2)) = \partial\hspace{0.5mm}\supertracebold\left(\dfrac{d\mathbb{A}_t}{dt}f'(\mathbb{A}_t^2)\right)$.
    \end{enumerate}
\end{theorem}

The next corollary is because the space of superconnections is affine \cite[Corollary 1.40]{bgv}.
\begin{corollary}
      For any $f\in\mathbb{C}[z]$ and any two primitive superconnections $\mathbb{A}_0$ and $\mathbb{A}_1$, let $\mathbb{A}_t = t\mathbb{A}_1 + (1-t)\mathbb{A}_0$, then we have
      $$\supertracebold(f(\mathbb{A}_1^2)) - \supertracebold(f(\mathbb{A}_0^2)) = \partial\int_0^1\supertracebold\left((\mathbb{A}_1-\mathbb{A}_0)f'(\mathbb{A}_t^2)\right)dt.$$
      This means that the primitive cohomology class of $\supertracebold(f(\mathbb{A}^2))$ is independent of $\mathbb{A}$.
\end{corollary}

    It could be very tempting to apply the classical transgression formula \cite[Proposition 1.41]{bgv} directly to each component of  
    $$\mathbb{A}_t^{2k}\in\Omega^{+}(M,\edmp(E))\oplus\Omega^{-}(M,\edmp(E))$$
    to verify Theorem \ref{main result}. 
    Unfortunately, because of the element $B\in\Omega^+(M,\edmp(E))$, neither of the components of $\mathbb{A}^{2k}$ is equal to the $2k$-th power of a superconnection. Thus, we cannot directly apply the classical transgression formula to each component.
    
    A simpler situation is when the de Rham cohomology class of the symplectic form $\omega$ is integral. In this situation, by \cite[Theorem 7.1]{tanaka_tseng_2018}, the chain complex (\ref{chain complex primitive}) is isomorphic to
    \begin{align*}
    \partial: \Omega^i(M)\oplus\left(\theta\wedge\Omega^{i-1}(M)\right)&\to\Omega^{i+1}(M)\oplus\left(\theta\wedge\Omega^i(M)\right)\ \ (0\leqslant i\leqslant \dim M+1)\\
    \alpha+\theta\wedge\beta&\mapsto d\alpha+\omega\wedge\beta-\theta\wedge\beta
    \end{align*}
    and computes the de Rham cohomology of a circle bundle $\pi: S\to M$. Here, $\theta$ is the angular form along the fiber direction of $S$.  
    In this situation, the primitive superconnection $\mathbb{A}$ is equal to $A+\theta B$, a superconnection on the superbundle $\pi^*E$. With $\mathbb{A} = A+\theta B$, Theorem \ref{main result} becomes an immediate corollary of \cite[Proposition 1.41]{bgv}. Therefore, Theorem \ref{main result} is a reasonable generalization of the classical transgression when the de Rham cohomology class of $\omega$ is non-integral. Also, the element $B\in\Omega^+(M,\edmp(E))$ has certain geometric meanings.

\begin{remark}\normalfont
    As in \cite[Definitions 1.1 \& 1.4]{tseng_zhou_symplectic_flat_functional_characteristic_classes2022} and \cite[(1.1) \& (1.11)]{tseng_and_zhou_2025mapping_yang_mills}, we can replace $\omega$ by a closed $2$-form and still prove Theorem \ref{main result}. However, due to the relation between $\omega$ and the angular form $\theta$ in the integral case, we prefer to present Theorem \ref{main result} using $\omega$. 
\end{remark}

We can now use $\supertracebold\left(f\left(\mathbb{A}^2\right)\right)$ to give the primitive versions of characteristic classes, which are represented by $\partial$-closed elements. According to \cite[Proposition 3.7]{tseng_and_zhou_symplectic_flat_connection_and_twisted_primitive2022}, the vanishing of $\mathbb{A}^2$ is relevant to K\"ahler-Einstein manifolds. For general $\mathbb{A}^2$, we hope the vanishing of primitive characteristic classes can bring us interesting findings. 

This paper is organized as follows. In Section \ref{section primitive bianchi identity}, we review superbundle-valued forms, explain signs and identifications, and prove the first half of Theorem \ref{main result}. In Section \ref{section transgression formula}, we prove the second half of Theorem \ref{main result}. In Section \ref{section char class}, we introduce the concept of primitive characteristic classes involving the symplectic structure on $M$ and propose a question about the relations between primitive characteristic classes and geometric information. 
 
\vspace{+3mm}
\noindent\textbf{Acknowledgments}. I want to thank \professor Xiaobo Liu, \professor Xiang Tang, \professor Li-Sheng Tseng, \professor Shanwen Wang, and \doctor Danhua Song for helpful discussions. Also, I want to thank Beijing International Center for Mathematical Research for providing an excellent working environment.

\section{Primitive Bianchi identity}\label{section primitive bianchi identity}
In this section, we review superbundle-valued forms
on $M$, explain signs and identifications, and prove the first half of Theorem \ref{main result} using the primitive version of the Bianchi identity.

For reviewing superbundle-valued forms and relevant sign conventions, we follow \cite[Sections 1.3-1.5]{bgv} and \cite[Sections 1.2-1.3]{weipingzhangnewedition}. We use ``$\edmp$'' and ``$\text{Hom}$'' to denote endomorphism bundles and homomorphism bundles respectively \cite[Section 1.2]{atiyah2018k}. For the $\mathbb{Z}_2$-graded superbundle $E = E^+\oplus E^-$ over $M$, according to the sign conventions in \cite[Section 1.3]{bgv}, we define
$$\edmp^+(E) \coloneqq \text{Hom}(E^+,E^+)\oplus\text{Hom}(E^-,E^-)$$
and 
$$\edmp^-(E) \coloneqq \text{Hom}(E^+,E^-)\oplus\text{Hom}(E^-,E^+).$$
Immediately, we see that $\edmp(E) = \edmp^+(E)\oplus\edmp^-(E)$. 
In addition, we define 
$$\Omega^+(M,E) \coloneqq \sum_{i\ \text{even}}\Omega^{i}(M,E^+)\oplus\sum_{i\ \text{odd}}\Omega^{i}(M,E^-)$$
and 
$$\Omega^-(M,E) \coloneqq \sum_{i\ \text{odd}}\Omega^{i}(M,E^+)\oplus\sum_{i\ \text{even}}\Omega^{i}(M,E^-). $$
A convenient way to understand $\Omega^+(M,E)$ and $\Omega^-(M,E)$ is via local expressions and total degrees. For any $\alpha\in\Omega^i(M)$, we let $$|\alpha| = i.$$ For any $v\in\Omega^0(M,E^+)$, (resp. $v\in\Omega^0(M,E^-)$), we let $$|v| = 0\ \text{(resp. $|v| = 1$)}.$$ Then, the total degree of $\alpha\otimes v$ is $$|\alpha\otimes v| \coloneqq |\alpha|+|v|.$$ When the total degree is even (resp. odd), $\alpha\otimes v$ is in $\Omega^+(M,E)$ (resp. $\Omega^-(M,E)$).

Given any $T\in\Omega(M,\edmp(E))$, we obtain a map
$$T: \Omega(M,E)\to\Omega(M,E).$$ One important thing to emphasize is that $T$ supercommutes with the elements in $\Omega(M)$. More precisely, if $T = \alpha\otimes L$ for some $\alpha\in\Omega(M)$ and $L\in\Omega^0(M,\edmp(E))$, then for any $\beta\in\Omega(M)$ and $v\in\Omega^0(M,E)$, $T$ maps $\beta\otimes v$ to
$$(-1)^{|L||\beta|}(\alpha\wedge\beta)\otimes L(v).$$
The sign appears here also affects the composition of two elements in $\Omega(M,\edmp(E))$. For example, given another $S = \beta\otimes K$ with $\beta\in\Omega(M)$ and $K\in\Omega^0(M,\edmp(E))$, the composition $ST$ is equal to 
$$(\beta\otimes K)(\alpha\otimes L) = (-1)^{|K||\alpha|}(\beta\wedge\alpha)\otimes(KL).$$
These rules of signs will be applied to computations like (\ref{A2k expression}) where there are compositions of elements in $\Omega(M,\edmp(E))$. 
\begin{remark}\normalfont
    We do not emphasize whether an element is homogeneous when the definitions and equations involving total degrees can be extended linearly to nonhomogeneous cases. 
\end{remark}

Given a differential operator $$D: \Omega(M,E)\to\Omega(M,E),$$ according to the decomposition $\Omega(M,E) = \Omega^+(M,E)\oplus\Omega^-(M,E)$, we write 
$$D = \begin{bmatrix}
    D_1 & D_2\\
    D_3 & D_4
\end{bmatrix}$$
and write any $\alpha\in\Omega(M,E)$ into $\alpha^++\alpha^-$. If for any $\beta\in\Omega(M)$, 
\begin{align}\label{supercommute differential operator version}
    D(\beta\wedge\alpha) = \beta\wedge D_1\alpha^+ + (-1)^{|\beta|}D_2\alpha^- + (-1)^{|\beta|}D_3\alpha^+ + \beta\wedge D_4\alpha^-, 
\end{align}
then by \cite[Section 1.4]{bgv}, we can identify $D$ with a unique element in $\Omega(M,\edmp(E))$. 
Similarly, for any linear map $$\mathbb{D}: \Omega(M,E)\oplus\Omega(M,E)\to \Omega(M,E)\oplus\Omega(M,E),$$ 
if we can find differential operators $D$ and $\tilde{D}$ on $\Omega(M,E)$ such that $D$ and $\tilde{D}$ both satisfy (\ref{supercommute differential operator version}), and such that $\mathbb{D}$ is given by
\begin{align}\label{operator D identified with a pair of endomorphisms}
\begin{split}
    \mathbb{D}: \Omega(M,E)\oplus\Omega(M,E)&\to\Omega(M,E)\oplus\Omega(M,E)\\
    (\alpha,\beta)&\mapsto \left(D\alpha, \tilde{D}\alpha+ D_1\beta^+ - D_2\beta^- - D_3\beta^+ + D_4\beta^-\right), 
\end{split}
\end{align}
then we identify $\mathbb{D}$ with the pair
\begin{align}\label{identifying endomorphisms}
    (D, \tilde{D})\in\Omega(M,\edmp(E))\oplus\Omega(M,\edmp(E)).
\end{align}
By Definition \ref{supertracebold definition}, for such a $\mathbb{D}$ given by (\ref{operator D identified with a pair of endomorphisms}), after the identification (\ref{identifying endomorphisms}), we let $$\supertracebold(\mathbb{D}) = \left(\supertrace(D),\supertrace(\tilde{D})\right)$$ and call it the supertrace of $\mathbb{D}$.

Recall the Definition \ref{extended superconnection} of the primitive superconnection
\begin{align*}
    \mathbb{A}: \Omega(M,E)\oplus\Omega(M,E)&\to \Omega(M,E)\oplus\Omega(M,E)\\
       (\alpha, \beta) &\mapsto (A\alpha+\omega\wedge\beta, B\alpha - A\beta)
\end{align*}
associated with a superconnection $A$ on $E$ and an element $B\in\Omega^+(M,\edmp(E))$.
 
\begin{remark}\label{not a connection remark}\normalfont 
    We still name $\mathbb{A}$ as a ``superconnection'' since when $$(\alpha,\beta)\in\Omega^+(M,E)\oplus\Omega^-(M,E)\ \ \text{(resp. $\Omega^-(M,E)\oplus\Omega^+(M,E)$)},$$ $\mathbb{A}$ maps $(\alpha,\beta)$ to $$(A\alpha+\omega\wedge\beta, B\alpha - A\beta)\in\Omega^{-}(M,E)\oplus\Omega^{+}(M,E)\ \ \text{(resp. $\Omega^+(M,E)\oplus\Omega^-(M,E)$)},$$ reversing the total degree of each component. However, we must emphasize that $\mathbb{A}$ is defined on pairs instead of on $E$-valued forms when the de Rham cohomology class of $\omega$ is not integral. 
\end{remark}

\begin{lemma}
    The even power $\mathbb{A}^{2k}$ of the primitive superconnection defines a unique element in $\Omega^{+}(M,\edmp(E))\oplus\Omega^{-}(M,\edmp(E))$. 
\end{lemma}
\begin{proof}
    For any $(\alpha,\beta)\in\Omega(M,E)\oplus\Omega(M,E)$, 
    we find 
    \begin{align}\label{A^2 computation}
         & \mathbb{A}^2(\alpha,\beta) \nonumber\\
        =\ & \mathbb{A}(A\alpha + \omega\wedge\beta, B\alpha-A\beta) \nonumber\\
        =\ & \left(A(A\alpha + \omega\wedge\beta)+\omega\wedge(B\alpha-A\beta), B(A\alpha + \omega\wedge\beta)-A(B\alpha-A\beta)\right) \nonumber\\
        =\ & (A^2\alpha+\omega\wedge B\alpha, A^2\beta+\omega\wedge B\beta+BA\alpha-AB\alpha). 
    \end{align}
    This gives us the unique $$(A^2+\omega\wedge B, BA-AB)\in\Omega^+(M,\edmp(E))\oplus\Omega^-(M,\edmp(E)).$$ 
    In general, we find that 
    \begin{align}\label{A2k expression}
        &\mathbb{A}^{2k}(\alpha,\beta) \nonumber \\
        =\ & \left((A^2+\omega\wedge B)^k\alpha,\ (A^2+\omega\wedge B)^k\beta+\sum_{i = 0}^{k-1}(A^2+\omega\wedge B)^i(BA-AB)(A^2+\omega\wedge B)^{k-1-i}\alpha\right).
    \end{align}
    By (\ref{identifying endomorphisms}), this identifies with the unique pair
    $$\left((A^2+\omega\wedge B)^k,\ \sum_{i = 0}^{k-1}(A^2+\omega\wedge B)^i(BA-AB)(A^2+\omega\wedge B)^{k-1-i}\right)$$
    in $\Omega^+(M,\edmp(E))\oplus\Omega^{-}(M,\edmp(E))$. 
\end{proof}

Using the identification, we have
\begin{align*}
    & \supertracebold(\mathbb{A}^{2k})\\
    =\ & \left(\supertrace\left((A^2+\omega\wedge B)^k\right),\ \supertrace\left(\sum_{i = 0}^{k-1}(A^2+\omega\wedge B)^i(BA-AB)(A^2+\omega\wedge B)^{k-1-i}\right)\right).
\end{align*}

According to \cite[Section 1.4]{bgv}, for any superconnection $A$ and any $\alpha\in\Omega(M,\edmp(E))$, 
we obtain the element $$[A,\alpha]\in\Omega^{k+1}(M,\edmp(E)).$$ The rule of this commutator is given in the way compatible with the Leibniz rule \cite[Definition 1.37]{bgv}: For $\alpha\in\Omega(M,\edmp(E))$, 
$$[A,\alpha] \coloneqq A\alpha - (-1)^{|\alpha|}\alpha A.$$
More precisely, for any $\eta\in\Omega(M,E)$, we have 
$$[A,\alpha]\eta = A(\alpha(\eta))-(-1)^{|\alpha|}\alpha(A\eta).$$
We extend this commutator to 
\begin{align}\label{definition of commutator for pairs}
    \left\llbracket\mathbb{A}, (\alpha,\beta)\right\rrbracket \coloneqq \left([A,\alpha]+\omega\wedge\beta,\ B\alpha-\alpha B-[A,\beta]\right)
\end{align}
for all $(\alpha,\beta)\in\Omega(M,\edmp(E))\oplus\Omega(M,\edmp(E))$. 

\begin{lemma}\label{boundary maps turns into bracket with primitive superconnection}
For any $(\alpha,\beta)\in\Omega(M,\edmp(E))\oplus\Omega(M,\edmp(E))$, we have 
    \begin{align*}
        \partial\hspace{+0.5mm}\supertracebold(\alpha,\beta) = \supertracebold\left\llbracket\mathbb{A},(\alpha,\beta)\right\rrbracket.
    \end{align*}
\end{lemma}
\begin{proof}
This is because 
\begin{align*}
       & \partial\hspace{+0.5mm}\supertracebold(\alpha,\beta)\\
    =\ & (d\supertrace(\alpha)+\omega\wedge\supertrace(\beta), -d\supertrace(\beta))\\
    =\ & (\supertrace[A,\alpha]+\omega\wedge\supertrace(\beta), -\supertrace[A,\beta])\ \ \text{(By \cite[Lemma 1.42]{bgv}.)}\\
    =\ & \supertracebold\left([A,\alpha]+\omega\wedge\beta,\ B\alpha-\alpha B-[A,\beta]\right) \ \ \text{(Since $\supertrace(B\alpha-\alpha B) = 0$.)}\\
    =\ & \supertracebold\left\llbracket\mathbb{A}, (\alpha,\beta)\right\rrbracket.
\end{align*}
The proof is complete. 
\end{proof}

We now prove the primitive Bianchi identity. 
\begin{proposition}\label{bianchi identity primitive version}
   For any $k\in\mathbb{N}$, we have $\left\llbracket\mathbb{A}, \mathbb{A}^{2k}\right\rrbracket = (0, 0)$. 
\end{proposition}
\begin{proof}
    By (\ref{A2k expression}) and (\ref{definition of commutator for pairs}), we find that 
    \begin{align*}
         & \left\llbracket\mathbb{A}, \mathbb{A}^{2k}\right\rrbracket\\
      =\ & \left(\left[A, (A^2+\omega\wedge B)^k\right]+\omega\wedge\sum_{i = 0}^{k-1}(A^2+\omega\wedge B)^i(BA-AB)(A^2+\omega\wedge B)^{k-1-i},\right. \\
       & \ \ \  B(A^2+\omega\wedge B)^k-(A^2+\omega\wedge B)^kB - A\sum_{i = 0}^{k-1}(A^2+\omega\wedge B)^i(BA-AB)(A^2+\omega\wedge B)^{k-1-i}\\
       & \ \ \ \ \ \ \ \ \ \ \ \ \ \ \ \ \ \ \ \ \ \ \ \ \ \ \ \ \ \ \ \ \ \ \ \ \ \ \ \ \ \ \ \ \ \ \ \ \ \ \left.-\sum_{i = 0}^{k-1}(A^2+\omega\wedge B)^i(BA-AB)(A^2+\omega\wedge B)^{k-1-i}A
    \right).
    \end{align*}
    Now, we use induction. When $k = 1$, the verification of the proposition is straightforward. Assume that for $k$, we have
    \begin{align}\label{induction assumption first}
        \left[A, (A^2+\omega\wedge B)^k\right]+\omega\wedge\sum_{i = 0}^{k-1}(A^2+\omega\wedge B)^i(BA-AB)(A^2+\omega\wedge B)^{k-1-i} = 0
    \end{align}
    and 
    \begin{align}\label{induction assumption second}
       & B(A^2+\omega\wedge B)^k-(A^2+\omega\wedge B)^kB - A\sum_{i = 0}^{k-1}(A^2+\omega\wedge B)^i(BA-AB)(A^2+\omega\wedge B)^{k-1-i} \nonumber\\
       & -\sum_{i = 0}^{k-1}(A^2+\omega\wedge B)^i(BA-AB)(A^2+\omega\wedge B)^{k-1-i}A \nonumber\\
       &=\ 0.
    \end{align}
    Then, for $k+1$, we find the first component of $\left\llbracket\mathbb{A}, \mathbb{A}^{2k+2}\right\rrbracket$ is equal to
    \begin{align*}
        & \left[A, (A^2+\omega\wedge B)^{k+1}\right]+\omega\wedge\sum_{i = 0}^{k}(A^2+\omega\wedge B)^i(BA-AB)(A^2+\omega\wedge B)^{k-i}\\
        =\ & A(A^2+\omega\wedge B)^{k+1}-(A^2+\omega\wedge B)^{k+1}A+\omega\wedge\sum_{i = 1}^{k-1}(A^2+\omega\wedge B)^i(BA-AB)(A^2+\omega\wedge B)^{k-i}\\
        & +\omega\wedge(BA-AB)(A^2+\omega\wedge B)^k+\omega\wedge(A^2+\omega\wedge B)^k(BA-AB)\\
        =\ & \left(A(A^2+\omega\wedge B)^k+\omega\wedge\sum_{i = 0}^{k-1}(A^2+\omega\wedge B)^i(BA-AB)(A^2+\omega\wedge B)^{k-i-1}\right)(A^2+\omega\wedge B)\\
        & -(A^2+\omega\wedge B)^{k+1}A+\omega\wedge(A^2+\omega\wedge B)^k(BA-AB)\\[2mm]
        & \text{(By (\ref{induction assumption first}) $\Rightarrow$)}\\
        =\ & (A^2+\omega\wedge B)^k A(A^2+\omega\wedge B)-(A^2+\omega\wedge B)^{k+1}A+\omega\wedge(A^2+\omega\wedge B)^k(BA-AB) \\[2mm]
        & \text{(since $\omega$ is a closed $2$-form\ $\Rightarrow$)}\\
        =\ & (A^2+\omega\wedge B)^k\left(A(A^2+\omega\wedge B)-(A^2+\omega\wedge B)A+\omega\wedge(BA-AB)\right) \ \\
        =\ & 0.
    \end{align*}
    We now compute the second component of $\left\llbracket\mathbb{A}, \mathbb{A}^{2k+2}\right\rrbracket$. It is equal to
    \begin{align*}
        & B(A^2+\omega\wedge B)^{k+1}-(A^2+\omega\wedge B)^{k+1}B\\
       & -A\sum_{i = 0}^{k}(A^2+\omega\wedge B)^i(BA-AB)(A^2+\omega\wedge B)^{k-i}\\
       & -\sum_{i = 0}^{k}(A^2+\omega\wedge B)^i(BA-AB)(A^2+\omega\wedge B)^{k-i}A\\
       =\ & \left(B(A^2+\omega\wedge B)^{k}-A\sum_{i = 0}^{k-1}(A^2+\omega\wedge B)^i(BA-AB)(A^2+\omega\wedge B)^{k-i-1}\right)(A^2+\omega\wedge B)\\
       & -(A^2+\omega\wedge B)^{k+1}B - A(A^2+\omega\wedge B)^k(BA-AB)\\
       & -\sum_{i = 0}^{k}(A^2+\omega\wedge B)^i(BA-AB)(A^2+\omega\wedge B)^{k-i}A\\[2mm]
       & \text{(By (\ref{induction assumption second})\ $\Rightarrow$)}\\
       =\ & \left(\sum_{i = 0}^{k-1}(A^2+\omega\wedge B)^i(BA-AB)(A^2+\omega\wedge B)^{k-i-1}A + (A^2+\omega\wedge B)^kB\right)(A^2+\omega\wedge B)\\
       & -(A^2+\omega\wedge B)^{k+1}B - A(A^2+\omega\wedge B)^k(BA-AB)\\
       & -\sum_{i = 0}^{k}(A^2+\omega\wedge B)^i(BA-AB)(A^2+\omega\wedge B)^{k-i}A \\
       =\ & \sum_{i = 0}^{k-1}(A^2+\omega\wedge B)^i(BA-AB)(A^2+\omega\wedge B)^{k-i-1}A(A^2+\omega\wedge B)\\
       & + (A^2+\omega\wedge B)^kB(A^2+\omega\wedge B) -(A^2+\omega\wedge B)^{k+1}B - A(A^2+\omega\wedge B)^k(BA-AB)\\
       & -\sum_{i = 0}^{k}(A^2+\omega\wedge B)^i(BA-AB)(A^2+\omega\wedge B)^{k-i}A\\
       =\ & \sum_{i = 0}^{k-1}(A^2+\omega\wedge B)^i(BA-AB)(A^2+\omega\wedge B)^{k-i-1}A(A^2+\omega\wedge B)\\
       & -\sum_{i = 0}^{k}(A^2+\omega\wedge B)^i(BA-AB)(A^2+\omega\wedge B)^{k-i}A\\
       & + (A^2+\omega\wedge B)^k\left(B(A^2+\omega\wedge B)-(A^2+\omega\wedge B)B\right) -A(A^2+\omega\wedge B)^k(BA-AB)\\
       =\ & \sum_{i = 0}^{k-1}(A^2+\omega\wedge B)^i(BA-AB)(A^2+\omega\wedge B)^{k-i-1}A(A^2+\omega\wedge B)\\
       & -\sum_{i = 0}^{k-1}(A^2+\omega\wedge B)^i(BA-AB)(A^2+\omega\wedge B)^{k-i}A\\
       & -(A^2+\omega\wedge B)^k(BA-AB)A + (A^2+\omega\wedge B)^k\left(BA^2-A^2B\right) - A(A^2+\omega\wedge B)^k(BA-AB)\\
       =\ & \sum_{i = 0}^{k-1}(A^2+\omega\wedge B)^i(BA-AB)(A^2+\omega\wedge B)^{k-i-1}\left(A(A^2+\omega\wedge B)-(A^2+\omega\wedge B)A\right)\\
       & -(A^2+\omega\wedge B)^k(BA-AB)A + (A^2+\omega\wedge B)^k\left(BA^2-A^2B\right) -A(A^2+\omega\wedge B)^k(BA-AB)\\
       =\ & \omega\wedge\sum_{i = 0}^{k-1}(A^2+\omega\wedge B)^i(BA-AB)(A^2+\omega\wedge B)^{k-i-1}(AB-BA)\\
       & -(A^2+\omega\wedge B)^k(BA-AB)A + (A^2+\omega\wedge B)^k\left(BA^2-A^2B\right) - A(A^2+\omega\wedge B)^k(BA-AB)\\[2mm]
       &\text{(By (\ref{induction assumption first})\ $\Rightarrow$)}\\
       =\ & -[A, (A^2+\omega\wedge B)^k](AB-BA)\\
       & -(A^2+\omega\wedge B)^k(BA-AB)A + (A^2+\omega\wedge B)^k\left(BA^2-A^2B\right) -A(A^2+\omega\wedge B)^k(BA-AB)\\
       =\ & (-A(A^2+\omega\wedge B)^k+(A^2+\omega\wedge B)^kA)(AB-BA)\\
       & -(A^2+\omega\wedge B)^k(BA^2-ABA) + (A^2+\omega\wedge B)^k\left(BA^2-A^2B\right) -A(A^2+\omega\wedge B)^k(BA-AB)\\
       =\ & (A^2+\omega\wedge B)^kA(AB-BA) - (A^2+\omega\wedge B)^k(BA^2-ABA) + (A^2+\omega\wedge B)^k\left(BA^2-A^2B\right)\\
       =\ & (A^2+\omega\wedge B)^k(A^2B-ABA-BA^2+ABA+BA^2-A^2B)\\
       =\ & 0.
    \end{align*}
    Thus, we finish the proof for the $k+1$ case and obtain the primitive Bianchi identity. 
\end{proof}

    By Proposition \ref{bianchi identity primitive version}, for any $k\in\mathbb{N}$, $$\partial\hspace{+0.5mm}\supertracebold(\mathbb{A}^{2k}) = \supertracebold\left\llbracket\mathbb{A},\mathbb{A}^{2k}\right\rrbracket = (0, 0).$$
    The proof of the first half of Theorem \ref{main result} is complete.

\section{Transgression formula}\label{section transgression formula}
In this section, we prove the second half of Theorem \ref{main result}, the transgression formula for any smooth family of primitive superconnections. We follow the procedure similar to that one proving \cite[Proposition 1.41, item 2]{bgv}. 

Given a smooth family of primitive superconnections
\begin{align*}
    \mathbb{A}_t: \Omega(M,E)\oplus\Omega(M,E)&\to\Omega(M,E)\oplus\Omega(M,E)\\
    (\alpha,\beta)&\mapsto (A_t\alpha+\omega\wedge\beta, B_t\alpha-A_t\beta), 
\end{align*}
its derivative with respect to $t$ is the map
\begin{align*}
    \dfrac{d\mathbb{A}_t}{dt}: \Omega(M,E)\oplus\Omega(M,E)&\to\Omega(M,E)\oplus\Omega(M,E)\\
    (\alpha,\beta)&\mapsto \left(\dfrac{dA_t}{dt}\alpha, \dfrac{dB_t}{dt}\alpha-\dfrac{dA_t}{dt}\beta\right).
\end{align*}
This identifies with the element
$$\left(\dfrac{dA_t}{dt}, \dfrac{dB_t}{dt}\right)\in\Omega^-(M,\edmp(E))\oplus\Omega^+(M,\edmp(E)).$$
In addition, for any $k\in\mathbb{N}$, the composition $\dfrac{d\mathbb{A}_t}{dt}\mathbb{A}_t^{2k}$ is given by
\begin{align}\label{dA_tA_t^{2k} expression}
       & \dfrac{d\mathbb{A}_t}{dt}\mathbb{A}_t^{2k}(\alpha,\beta) \nonumber \\
    = & \dfrac{d\mathbb{A}_t}{dt}\left((A_t^2+\omega\wedge B_t)^k\alpha,\ (A_t^2+\omega\wedge B_t)^k\beta+\sum_{i = 0}^{k-1}(A_t^2+\omega\wedge B_t)^i(B_tA_t-A_tB_t)(A_t^2+\omega\wedge B_t)^{k-1-i}\alpha\right) \nonumber \\
    = & \Bigg(\dfrac{dA_t}{dt}(A_t^2+\omega\wedge B_t)^k\alpha, \nonumber \\
    & \dfrac{dB_t}{dt}(A_t^2+\omega\wedge B_t)^k\alpha - \dfrac{dA_t}{dt}\left(\hspace{-1mm}(A_t^2+\omega\wedge B_t)^k\beta+\sum_{i = 0}^{k-1}(A_t^2+\omega\wedge B_t)^i(B_tA_t-A_tB_t)(A_t^2+\omega\wedge B_t)^{k-1-i}\alpha\right)\hspace{-1.5mm}\Bigg)
\end{align}
for all $(\alpha,\beta)\in\Omega(M,E)\oplus\Omega(M,E)$.
This identifies with an element 
$$\left(\dfrac{dA_t}{dt}(A_t^2+\omega\wedge B_t)^k,\ \dfrac{dB_t}{dt}(A_t^2+\omega\wedge B_t)^k-\dfrac{dA_t}{dt}\sum_{i = 0}^{k-1}(A_t^2+\omega\wedge B_t)^i(B_tA_t-A_tB_t)(A_t^2+\omega\wedge B_t)^{k-1-i}\right)$$
in $$\Omega^{-}(M,\edmp(E))\oplus\Omega^{+}(M,\edmp(E)).$$
Accordingly, we obtain $$\supertracebold\left(\dfrac{d\mathbb{A}_t}{dt}\mathbb{A}_t^{2k}\right).$$

At present, to prove the transgression formula, we only need to show 
$$\dfrac{1}{k+1}\dfrac{d}{dt}\supertracebold\left(\mathbb{A}_t^{2k+2}\right) = \partial\hspace{0.5mm}\supertracebold\left(\dfrac{d\mathbb{A}_t}{dt}\mathbb{A}_t^{2k}\right).$$
According to Lemma \ref{boundary maps turns into bracket with primitive superconnection}, it is equivalent to show 
    \begin{align*}
        \dfrac{1}{k+1}\dfrac{d}{dt}\supertracebold\left(\mathbb{A}_t^{2k+2}\right) = \supertracebold\left\llbracket\mathbb{A}_t, \dfrac{d\mathbb{A}_t}{dt}\mathbb{A}_t^{2k}\right\rrbracket.
    \end{align*}
By (\ref{A2k expression}), we find 
\begin{align*}
& \dfrac{d}{dt}\supertracebold\left(\mathbb{A}_t^{2k+2}\right)\\
=\ & \left(\supertrace\left(\dfrac{d}{dt}(A_t^2+\omega\wedge B_t)^{k+1}\right),\ \dfrac{d}{dt}\supertrace\left(\sum_{i = 0}^{k-1}(A_t^2+\omega\wedge B_t)^i(B_tA_t-A_tB_t)(A_t^2+\omega\wedge B_t)^{k-1-i}\right)\right)\\
=\ & \left(\supertrace\left(\dfrac{d}{dt}(A_t^2+\omega\wedge B_t)^{k+1}\right),\ \dfrac{d}{dt}\supertrace\left((k+1)(A_t^2+\omega\wedge B_t)^k(B_tA_t-A_tB_t)\right)\right)\\
=\ & \left(\supertrace\left(\dfrac{d}{dt}(A_t^2+\omega\wedge B_t)^{k+1}\right),\ \supertrace\left((k+1)\dfrac{d}{dt}\left((A_t^2+\omega\wedge B_t)^k(B_tA_t-A_tB_t)\right)\right)\right).
\end{align*}
In addition, by (\ref{definition of commutator for pairs}), the first component of $\left\llbracket\mathbb{A}_t, \dfrac{d\mathbb{A}_t}{dt}\mathbb{A}_t^{2k}\right\rrbracket$ is equal to
\begin{align*}
    & A_t\dfrac{dA_t}{dt}(A_t^2+\omega\wedge B_t)^k+\dfrac{dA_t}{dt}(A_t^2+\omega\wedge B_t)^kA_t\\
    & + \omega\wedge\dfrac{dB_t}{dt}(A_t^2+\omega\wedge B_t)^k-\omega\wedge\dfrac{dA_t}{dt}\sum_{i = 0}^{k-1}(A_t^2+\omega\wedge B_t)^i(B_tA_t-A_tB_t)(A_t^2+\omega\wedge B_t)^{k-1-i}.
\end{align*}
We first need to verify that 
\begin{align}\label{first component want to show in supertrace}
    & \dfrac{1}{k+1}\supertrace\left(\dfrac{d}{dt}(A_t^2+\omega\wedge B_t)^{k+1}\right) \nonumber\\
    =\ & \supertrace\Bigg(A_t\dfrac{dA_t}{dt}(A_t^2+\omega\wedge B_t)^k+\dfrac{dA_t}{dt}(A_t^2+\omega\wedge B_t)^kA_t \nonumber\\
    & + \omega\wedge\dfrac{dB_t}{dt}(A_t^2+\omega\wedge B_t)^k-\omega\wedge\dfrac{dA_t}{dt}\sum_{i = 0}^{k-1}(A_t^2+\omega\wedge B_t)^i(B_tA_t-A_tB_t)(A_t^2+\omega\wedge B_t)^{k-1-i}\Bigg).
\end{align}
The left hand side of (\ref{first component want to show in supertrace}) is equal to 
\begin{align*}
     & \supertrace\left(\dfrac{d(A_t^2+\omega\wedge B_t)}{dt}(A_t^2+\omega\wedge B_t)^k\right)\\
    =\ & \supertrace\left(\left(A_t\dfrac{dA_t}{dt}+\dfrac{dA_t}{dt}A_t+\omega\wedge\dfrac{dB_t}{dt}\right)(A_t^2+\omega\wedge B_t)^k\right).
\end{align*}
Thus, (\ref{first component want to show in supertrace}) is equivalent to
\begin{align*}
     & \supertrace\left(\left(A_t\dfrac{dA_t}{dt}+\dfrac{dA_t}{dt}A_t+\omega\wedge\dfrac{dB_t}{dt}\right)(A_t^2+\omega\wedge B_t)^k\right)\\
    =\ & \supertrace\Bigg(A_t\dfrac{dA_t}{dt}(A_t^2+\omega\wedge B_t)^k+\dfrac{dA_t}{dt}(A_t^2+\omega\wedge B_t)^kA_t\\
    & + \omega\wedge\dfrac{dB_t}{dt}(A_t^2+\omega\wedge B_t)^k-\omega\wedge\dfrac{dA_t}{dt}\sum_{i = 0}^{k-1}(A_t^2+\omega\wedge B_t)^i(B_tA_t-A_tB_t)(A_t^2+\omega\wedge B_t)^{k-1-i}\Bigg)
\end{align*}
$\Leftrightarrow$
\begin{align*}
     & \supertrace\left(\dfrac{dA_t}{dt}A_t(A_t^2+\omega\wedge B_t)^k\right)\\
    =\ & \supertrace\left(\dfrac{dA_t}{dt}(A_t^2+\omega\wedge B_t)^kA_t-\omega\wedge\dfrac{dA_t}{dt}\sum_{i = 0}^{k-1}(A_t^2+\omega\wedge B_t)^i(B_tA_t-A_tB_t)(A_t^2+\omega\wedge B_t)^{k-1-i}\right)
\end{align*}
$\Leftrightarrow$
\begin{align*}
   0 =\ &\supertrace\Bigg(\dfrac{dA_t}{dt}A_t(A_t^2+\omega\wedge B_t)^k-\dfrac{dA_t}{dt}(A_t^2+\omega\wedge B_t)^kA_t\\
    & \ \ \ \ \ \ +\omega\wedge\dfrac{dA_t}{dt}\sum_{i = 0}^{k-1}(A_t^2+\omega\wedge B_t)^i(B_tA_t-A_tB_t)(A_t^2+\omega\wedge B_t)^{k-1-i}\Bigg).
\end{align*}
The last equation exactly repeats (\ref{induction assumption first}), and thus (\ref{first component want to show in supertrace}) is true. 

The second component of $$\left\llbracket\mathbb{A}_t, \dfrac{d\mathbb{A}_t}{dt}f'(\mathbb{A}_t^2)\right\rrbracket$$
is equal to 
\begin{align*}
&B_t\dfrac{dA_t}{dt}(A_t^2+\omega\wedge B_t)^k - \dfrac{dA_t}{dt}(A_t^2+\omega\wedge B_t)^kB_t\\
    & -A_t\dfrac{dB_t}{dt}(A_t^2+\omega\wedge B_t)^k + A_t\dfrac{dA_t}{dt}\sum_{i = 0}^{k-1}(A_t^2+\omega\wedge B_t)^i(B_tA_t-A_tB_t)(A_t^2+\omega\wedge B_t)^{k-1-i} \\
    & +\dfrac{dB_t}{dt}(A_t^2+\omega\wedge B_t)^kA_t - \dfrac{dA_t}{dt}\sum_{i = 0}^{k-1}(A_t^2+\omega\wedge B_t)^i(B_tA_t-A_tB_t)(A_t^2+\omega\wedge B_t)^{k-1-i}A_t.
\end{align*}
By the property \cite[Definition 1.30, Proposition 1.31]{bgv} of supertraces, the supertrace of the second component is equal to  
\begin{align*}
    & \supertrace\Bigg(-A_t\dfrac{dB_t}{dt}(A_t^2+\omega\wedge B_t)^k + A_t\dfrac{dA_t}{dt}\sum_{i = 0}^{k-1}(A_t^2+\omega\wedge B_t)^i(B_tA_t-A_tB_t)(A_t^2+\omega\wedge B_t)^{k-1-i} \\
    & \ \ \ \ \ \ +\dfrac{dB_t}{dt}(A_t^2+\omega\wedge B_t)^kA_t - \dfrac{dA_t}{dt}\sum_{i = 0}^{k-1}(A_t^2+\omega\wedge B_t)^i(B_tA_t-A_tB_t)(A_t^2+\omega\wedge B_t)^{k-1-i}A_t\Bigg).
\end{align*}
We now need to verify that
\begin{align}\label{final part in the transgression to show}
    & \supertrace\Bigg(-A_t\dfrac{dB_t}{dt}(A_t^2+\omega\wedge B_t)^k + A_t\dfrac{dA_t}{dt}\sum_{i = 0}^{k-1}(A_t^2+\omega\wedge B_t)^i(B_tA_t-A_tB_t)(A_t^2+\omega\wedge B_t)^{k-1-i} \nonumber\\
    & \ \ \ \ \ \ +\dfrac{dB_t}{dt}(A_t^2+\omega\wedge B_t)^kA_t - \dfrac{dA_t}{dt}\sum_{i = 0}^{k-1}(A_t^2+\omega\wedge B_t)^i(B_tA_t-A_tB_t)(A_t^2+\omega\wedge B_t)^{k-1-i}A_t\Bigg) \nonumber\\
    =\ & \supertrace\left(\dfrac{d}{dt}\left((A_t^2+\omega\wedge B_t)^k(B_tA_t-A_tB_t)\right)\right).
\end{align}
The right hand side of (\ref{final part in the transgression to show}) is equal to 
\begin{align*}
    & \supertrace\Bigg(\sum_{i = 0}^{k-1}(A_t^2+\omega\wedge B_t)^i\dfrac{d(A_t^2+\omega\wedge B_t)}{dt}(A_t^2+\omega\wedge B_t)^{k-i-1}(B_tA_t-A_tB_t)\\
    & \ \ \ \ \ \ \ \ + (A_t^2+\omega\wedge B_t)^k\dfrac{d(B_tA_t-A_tB_t)}{dt}\Bigg)\\
    =\ & \supertrace\Bigg(\sum_{i = 0}^{k-1}\dfrac{d(A_t^2+\omega\wedge B_t)}{dt}(A_t^2+\omega\wedge B_t)^{k-i-1}(B_tA_t-A_tB_t)(A_t^2+\omega\wedge B_t)^i\\
    & \ \ \ \ \ \ \ \ + (A_t^2+\omega\wedge B_t)^k\dfrac{d(B_tA_t-A_tB_t)}{dt}\Bigg)\\
    =\ & \supertrace\Bigg(\sum_{i = 0}^{k-1}\dfrac{d(A_t^2+\omega\wedge B_t)}{dt}(A_t^2+\omega\wedge B_t)^{i}(B_tA_t-A_tB_t)(A_t^2+\omega\wedge B_t)^{k-i-1}\\
    & \ \ \ \ \ \ \ \ + (A_t^2+\omega\wedge B_t)^k\dfrac{d(B_tA_t-A_tB_t)}{dt}\Bigg)\\
    =\ & \supertrace\Bigg(A_t\dfrac{dA_t}{dt}\sum_{i = 0}^{k-1}(A_t^2+\omega\wedge B_t)^{i}(B_tA_t-A_tB_t)(A_t^2+\omega\wedge B_t)^{k-i-1}\\
    & \ \ \ \ \ \ \ \ + \dfrac{dA_t}{dt}A_t\sum_{i = 0}^{k-1}(A_t^2+\omega\wedge B_t)^{i}(B_tA_t-A_tB_t)(A_t^2+\omega\wedge B_t)^{k-i-1}\\
    & \ \ \ \ \ \ \ \ + \omega\wedge\dfrac{dB_t}{dt}\sum_{i = 0}^{k-1}(A_t^2+\omega\wedge B_t)^{i}(B_tA_t-A_tB_t)(A_t^2+\omega\wedge B_t)^{k-i-1}\\
    & \ \ \ \ \ \ \ \ + (A_t^2+\omega\wedge B_t)^k\dfrac{dB_t}{dt}A_t + (A_t^2+\omega\wedge B_t)^k B_t\dfrac{dA_t}{dt}\\
    & \ \ \ \ \ \ \ \ - (A_t^2+\omega\wedge B_t)^k\dfrac{dA_t}{dt}B_t - (A_t^2+\omega\wedge B_t)^k A_t\dfrac{dB_t}{dt}\Bigg).
\end{align*}
Thus, the verification of (\ref{final part in the transgression to show}) is equivalent to verifying 
\begin{align*}
    & \supertrace\Bigg(-A_t\dfrac{dB_t}{dt}(A_t^2+\omega\wedge B_t)^k + A_t\dfrac{dA_t}{dt}\sum_{i = 0}^{k-1}(A_t^2+\omega\wedge B_t)^i(B_tA_t-A_tB_t)(A_t^2+\omega\wedge B_t)^{k-1-i} \\
    & \ \ \ \ \ \ +\dfrac{dB_t}{dt}(A_t^2+\omega\wedge B_t)^kA_t - \dfrac{dA_t}{dt}\sum_{i = 0}^{k-1}(A_t^2+\omega\wedge B_t)^i(B_tA_t-A_tB_t)(A_t^2+\omega\wedge B_t)^{k-1-i}A_t\Bigg)\\
=\ & \supertrace\Bigg(A_t\dfrac{dA_t}{dt}\sum_{i = 0}^{k-1}(A_t^2+\omega\wedge B_t)^{i}(B_tA_t-A_tB_t)(A_t^2+\omega\wedge B_t)^{k-i-1}\\
    & \ \ \ \ \ \ \ \ + \dfrac{dA_t}{dt}A_t\sum_{i = 0}^{k-1}(A_t^2+\omega\wedge B_t)^{i}(B_tA_t-A_tB_t)(A_t^2+\omega\wedge B_t)^{k-i-1}\\
    & \ \ \ \ \ \ \ \ + \omega\wedge\dfrac{dB_t}{dt}\sum_{i = 0}^{k-1}(A_t^2+\omega\wedge B_t)^{i}(B_tA_t-A_tB_t)(A_t^2+\omega\wedge B_t)^{k-i-1}\\
    & \ \ \ \ \ \ \ \ + (A_t^2+\omega\wedge B_t)^k\dfrac{dB_t}{dt}A_t + (A_t^2+\omega\wedge B_t)^k B_t\dfrac{dA_t}{dt}\\
    & \ \ \ \ \ \ \ \ - (A_t^2+\omega\wedge B_t)^k\dfrac{dA_t}{dt}B_t - (A_t^2+\omega\wedge B_t)^k A_t\dfrac{dB_t}{dt}\Bigg).
\end{align*}
This is equivalent to verify 
\begin{align*}
    & \supertrace\Bigg(-A_t\dfrac{dB_t}{dt}(A_t^2+\omega\wedge B_t)^k\\
    & \ \ \ \ \ \ +\dfrac{dB_t}{dt}(A_t^2+\omega\wedge B_t)^kA_t - \dfrac{dA_t}{dt}\sum_{i = 0}^{k-1}(A_t^2+\omega\wedge B_t)^i(B_tA_t-A_tB_t)(A_t^2+\omega\wedge B_t)^{k-1-i}A_t\Bigg)\\
=\ & \supertrace\Bigg(\dfrac{dA_t}{dt}A_t\sum_{i = 0}^{k-1}(A_t^2+\omega\wedge B_t)^{i}(B_tA_t-A_tB_t)(A_t^2+\omega\wedge B_t)^{k-i-1}\\
    & \ \ \ \ \ \ \ \ + \omega\wedge\dfrac{dB_t}{dt}\sum_{i = 0}^{k-1}(A_t^2+\omega\wedge B_t)^{i}(B_tA_t-A_tB_t)(A_t^2+\omega\wedge B_t)^{k-i-1}\\
    & \ \ \ \ \ \ \ \ + (A_t^2+\omega\wedge B_t)^k\dfrac{dB_t}{dt}A_t + (A_t^2+\omega\wedge B_t)^k B_t\dfrac{dA_t}{dt}\\
    & \ \ \ \ \ \ \ \ - (A_t^2+\omega\wedge B_t)^k\dfrac{dA_t}{dt}B_t - (A_t^2+\omega\wedge B_t)^k A_t\dfrac{dB_t}{dt}\Bigg).
\end{align*}
Using (\ref{induction assumption first}) and (\ref{induction assumption second}), now it is equivalent to prove
\begin{align*}
       & \supertrace\left(-A_t\dfrac{dB_t}{dt}(A_t^2+\omega\wedge B_t)^k +\dfrac{dB_t}{dt}(A_t^2+\omega\wedge B_t)^kA_t\right) \\
    =\ & \supertrace\Bigg(\dfrac{dA_t}{dt}A_t\sum_{i = 0}^{k-1}(A_t^2+\omega\wedge B_t)^{i}(B_tA_t-A_tB_t)(A_t^2+\omega\wedge B_t)^{k-i-1}\\
    & \ \ \ \ \ \ \ \ + \dfrac{dA_t}{dt}\sum_{i = 0}^{k-1}(A_t^2+\omega\wedge B_t)^i(B_tA_t-A_tB_t)(A_t^2+\omega\wedge B_t)^{k-1-i}A_t\\
       & \ \ \ \ \ \ \ \ + \dfrac{dB_t}{dt}\left((A_t^2+\omega\wedge B_t)^k A_t-A_t(A_t^2+\omega\wedge B_t)^k\right)\\
       & \ \ \ \ \ \ \ \ + (A_t^2+\omega\wedge B_t)^k\dfrac{dB_t}{dt}A_t + (A_t^2+\omega\wedge B_t)^k B_t\dfrac{dA_t}{dt}\\
       & \ \ \ \ \ \ \ \ - (A_t^2+\omega\wedge B_t)^k\dfrac{dA_t}{dt}B_t - (A_t^2+\omega\wedge B_t)^k A_t\dfrac{dB_t}{dt}\Bigg)
\end{align*}
$\Leftrightarrow$
\begin{align*}
       & \supertrace\left(-A_t\dfrac{dB_t}{dt}(A_t^2+\omega\wedge B_t)^k +\dfrac{dB_t}{dt}(A_t^2+\omega\wedge B_t)^kA_t\right) \\
    =\ & \supertrace\Bigg(\dfrac{dA_t}{dt}B_t(A_t^2+\omega\wedge B_t)^k -  \dfrac{dA_t}{dt}(A_t^2+\omega\wedge B_t)^k B_t\\
       & \ \ \ \ \ \ \ \ + \dfrac{dB_t}{dt}(A_t^2+\omega\wedge B_t)^k A_t-\dfrac{dB_t}{dt}A_t(A_t^2+\omega\wedge B_t)^k\\
       & \ \ \ \ \ \ \ \ + (A_t^2+\omega\wedge B_t)^k\dfrac{dB_t}{dt}A_t + (A_t^2+\omega\wedge B_t)^k B_t\dfrac{dA_t}{dt}\\
       & \ \ \ \ \ \ \ \ - (A_t^2+\omega\wedge B_t)^k\dfrac{dA_t}{dt}B_t - (A_t^2+\omega\wedge B_t)^k A_t\dfrac{dB_t}{dt}\Bigg)
\end{align*}
$\Leftrightarrow$ $\left(\text{This step is because $\supertrace\left(-\dfrac{dA_t}{dt}(A_t^2+\omega\wedge B_t)^k B_t+ (A_t^2+\omega\wedge B_t)^k B_t\dfrac{dA_t}{dt}\right)$} = 0.\right)$
\begin{align*}
    & \supertrace\left(-A_t\dfrac{dB_t}{dt}(A_t^2+\omega\wedge B_t)^k\right) \\
    =\ & \supertrace\left((A_t^2+\omega\wedge B_t)^k\dfrac{dB_t}{dt}A_t -  \dfrac{dB_t}{dt}A_t(A_t^2+\omega\wedge B_t)^k - (A_t^2+\omega\wedge B_t)^k A_t\dfrac{dB_t}{dt}\right)
\end{align*}
$\Leftrightarrow$
\begin{align*}
   \supertrace\left((A_t^2+\omega\wedge B_t)^k\dfrac{dB_t}{dt}A_t -  \dfrac{dB_t}{dt}A_t(A_t^2+\omega\wedge B_t)^k - (A_t^2+\omega\wedge B_t)^k A_t\dfrac{dB_t}{dt}+A_t\dfrac{dB_t}{dt}(A_t^2+\omega\wedge B_t)^k\right) = 0
\end{align*}
$\Leftrightarrow$
\begin{align*}
   \supertrace\left((A_t^2+\omega\wedge B_t)^k\left(\dfrac{dB_t}{dt}A_t-A_t\dfrac{dB_t}{dt}\right) - \left(\dfrac{dB_t}{dt}A_t-A_t\dfrac{dB_t}{dt}\right)(A_t^2+\omega\wedge B_t)^k\right) = 0.
\end{align*}
The last equation is true because of \cite[Definition 1.30, Proposition 1.31]{bgv}. Thus, (\ref{final part in the transgression to show}) holds true. The proof of the second half of Theorem \ref{main result} is complete.

\section{Characteristic classes involving $\omega$}\label{section char class}
In this section, we introduce the primitive versions of some characteristic classes. Their constructions involve the information from both the superbundle $E$ and the symplectic form $\omega$. Also, we propose a possible further direction. 

\begin{example}\label{example 1}\normalfont
    Following \cite[(1.30)]{bgv}, we choose $f(z) = e^{-z}$ in Theorem \ref{main result}. Then, we obtain
$$\text{ch}(\mathbb{A},\omega)\coloneqq\supertrace\left(e^{-\mathbb{A}^2}\right)$$
for any primitive superconnection $\mathbb{A}$. Since $\dim M<\infty$, we have $\mathbb{A}^{2k} = (0,0)$ when $2k>\dim M$. Therefore, the power series 
$$e^{-z} = 1-z+\dfrac{1}{2!}z^2-\dfrac{1}{3!}z^3+\cdots$$
becomes a finite sum after we plug $\mathbb{A}^2$ into $e^{-z}$, and $\text{ch}(\mathbb{A},\omega)$ is well-defined. 
\end{example}
\begin{remark}\normalfont
    Similar to \cite[Section 1.6.5]{wittendeformationweipingzhang}, using $\text{ch}(\mathbb{A},\omega)$, the readers may expect the realization of the primitive version of the Chern-Simons form. In addition, when the de Rham cohomology class of $\omega$ is integral, this Chern-Simons form should recover the Chern-Simons form \cite[Section 4]{tseng_zhou_symplectic_flat_functional_characteristic_classes2022} on the associated circle bundle over $M$. 
\end{remark}

Besides the Chern character, we can also define other primitive characteristic classes like the primitive version of the A-hat genus. 
\begin{example}\label{example 2}\normalfont
Suppose that $E^- = 0$, $B\in\Omega^0(M,\edmp(E))$, and
$$\nabla: \Omega^i(M,E)\to\Omega^{i+1}(M,E)$$ is a Koszul connection \cite[Definition 1.4]{wittendeformationweipingzhang} on $E$. Then, in this situation, we let
\begin{align*}
    \mathbb{A}: \Omega^{i}(M,E)\oplus\Omega^{i-1}(M,E)&\to\Omega^{i+1}(M,E)\oplus\Omega^{i}(M,E)\\
    (\alpha,\beta)&\mapsto (\nabla\alpha + \omega\wedge\beta, B\alpha-\nabla\beta), 
\end{align*}
and the supertrace $\supertracebold$ becomes trace $\tracebold$. Following \cite[(1.35)]{bgv}, we choose $$f(z) = \dfrac{1}{2}\ln\left(\dfrac{z/2}{\sinh(z/2)}\right)$$
and then obtain the primitive cohomology class represented by 
$$\tracebold\left(\dfrac{1}{2}\ln\left(\dfrac{\mathbb{A}^2/2}{\sinh(\mathbb{A}^2/2)}\right)\right).$$
Immediately, the primitive A-hat genus is defined as
    $$\hat{\text{A}}(\mathbb{A},\omega) = \exp\left(\tracebold\left(\dfrac{1}{2}\ln\left(\dfrac{\mathbb{A}^2/2}{\sinh(\mathbb{A}^2/2)}\right)\right)\right).$$  
When using the power series of $\exp$, we need the product structure on $\Omega(M)\oplus\Omega(M)$, i.e., 
\begin{align*}
    (\alpha_1, \beta_1)\cdot (\alpha_2, \beta_2) \coloneqq (\alpha_1\wedge\alpha_2, \beta_1\wedge\alpha_2+(-1)^{|\alpha_1|}\alpha_1\wedge\beta_2)
\end{align*}
for any $(\alpha_1, \beta_1), (\alpha_2, \beta_2)\in\Omega(M)\oplus\Omega(M).$
\end{example}

    Like in \cite[Sections 1.5-1.6]{wittendeformationweipingzhang}, we may also replace $\mathbb{A}$ by $\dfrac{\sqrt{-1}}{2\pi}\mathbb{A}$ in the definitions of the above primitive characteristic classes. Also, besides Examples \ref{example 1} and \ref{example 2}, for other characteristic classes presented in \cite[Section 1.6]{wittendeformationweipingzhang}, we can try to construct their associated primitive versions as well.

Finally, we propose a further direction. According to \cite[Proposition 3.7]{tseng_and_zhou_symplectic_flat_connection_and_twisted_primitive2022}, if $M$ is K\"ahler equipped with the compatible metric, and if the associated Levi-Civita connection is symplectically flat, then $M$ is K\"ahler-Einstein. Equivalently, when $\mathbb{A}^2$ vanishes, $M$ is K\"ahler-Einstein. Now, if $\mathbb{A}^2$ does not vanish, but some primitive characteristic class constructed by $\mathbb{A}^2$ vanishes, we hope to obtain some conclusions related to geometric objects. 

We end this paper by summarizing the above idea into the following question: 
\begin{question}\normalfont
    For the primitive versions of Characteristic classes, what geometric information do they provide or obstruct? 
\end{question}

\bibliographystyle{abbrv}
\bibliography{mybib.bib}

\begin{thebibliography}{10}

\bibitem{atiyah2018k}
M.~F. Atiyah.
\newblock {\em K-theory}.
\newblock CRC Press, 2018.

\bibitem{bgv}
N.~Berline, E.~Getzler, and M.~Vergne.
\newblock {\em Heat Kernels and Dirac Operators}.
\newblock Springer, 2004.

\bibitem{tangtsengclausensymplecticwitten}
D.~Clausen, X.~Tang, and L.-S. Tseng.
\newblock Symplectic {M}orse theory and {W}itten deformation. {P}reprint ar{X}iv:2211.11712v4, 2022.

\bibitem{tangtsengclausenmappingcone}
D.~Clausen, X.~Tang, and L.-S. Tseng.
\newblock Mapping cone and {M}orse theory. {P}reprint ar{X}iv:2405.02272v2, 2024.

\bibitem{quillen_superconnection}
D.~Quillen.
\newblock Superconnections and the {C}hern character.
\newblock {\em Topology}, 24(1):89--95, 1985.

\bibitem{tanaka_tseng_2018}
H.~L. Tanaka and L.-S. Tseng.
\newblock Odd sphere bundles, symplectic manifolds, and their intersection theory.
\newblock {\em Cambridge Journal of Mathematics}, 6(3):213--266, 2018.

\bibitem{tty3rd}
C.-J. Tsai, L.-S. Tseng, and S.-T. Yau.
\newblock {Cohomology and Hodge theory on symplectic manifolds: III}.
\newblock {\em Journal of Differential Geometry}, 103(1):83--143, 2016.

\bibitem{tty1st}
L.-S. Tseng and S.-T. Yau.
\newblock {Cohomology and Hodge theory on symplectic manifolds: I}.
\newblock {\em Journal of Differential Geometry}, 91(3):383--416, 2012.

\bibitem{tty2nd}
L.-S. Tseng and S.-T. Yau.
\newblock {Cohomology and Hodge theory on symplectic manifolds: II}.
\newblock {\em Journal of Differential Geometry}, 91(3):417--443, 2012.

\bibitem{tseng_and_zhou_symplectic_flat_connection_and_twisted_primitive2022}
L.-S. Tseng and J.~Zhou.
\newblock Symplectic flatness and twisted primitive cohomology.
\newblock {\em The Journal of Geometric Analysis}, 32(11):282, 2022.

\bibitem{tseng_zhou_symplectic_flat_functional_characteristic_classes2022}
L.-S. Tseng and J.~Zhou.
\newblock Symplectically flat connections and their functionals. {P}reprint ar{X}iv:2210.03032v5, 2022.

\bibitem{tseng_and_zhou_2025mapping_yang_mills}
L.-S. Tseng and J.~Zhou.
\newblock Mapping cone connections and their {Yang--Mills functional}.
\newblock {\em Communications in Mathematical Physics}, 406(7):156, 2025.

\bibitem{wittendeformationweipingzhang}
W.~Zhang.
\newblock {\em Lectures on Chern-Weil Theory and Witten Deformations}.
\newblock World Scientific, 2001.

\bibitem{weipingzhangnewedition}
W.~Zhang and H.~Feng.
\newblock {\em Geometry and Analysis on Manifolds}.
\newblock Higher Education Press, {C}hinese edition, 2022.

\bibitem{symplectic_semi_char_2025}
H.~Zhuang.
\newblock Symplectic semi-characteristics. {P}reprint ar{X}iv:2505.14496v1, 2025.

\bibitem{1_filtered_semi_char_2025}
H.~Zhuang.
\newblock A vanishing property about the $1$-filtered cohomology groups of $(4n+2)$-dimensional closed symplectic manifolds. {P}reprint ar{X}iv:2510.10630v1, 2025.

\end{thebibliography}
\end{document}